\documentclass[a4paper,11pt]{article}
\usepackage[dvipsnames]{xcolor}
\usepackage{amsthm,amssymb,amsmath,amscd,empheq,mathrsfs, color, stmaryrd}
\usepackage[all]{xy}
\usepackage{url}
\usepackage{todonotes}
\usepackage[margin=1.0in]{geometry}
\usepackage[shortlabels]{enumitem}
\usepackage[utf8]{inputenc}
\usepackage[T1]{fontenc}
\usepackage[colorlinks=true,hyperindex, linkcolor=BrickRed, pagebackref=false, citecolor=green,pdfpagelabels, urlcolor=NavyBlue]{hyperref}
\usepackage[capitalize]{cleveref}
\usepackage{relsize}
\usepackage[bbgreekl]{mathbbol}
\usepackage{amsfonts}
\usepackage{quiver}
\DeclareSymbolFontAlphabet{\mathbb}{AMSb}
\DeclareSymbolFontAlphabet{\mathbbl}{bbold}
\newcommand{\Prism}{{\mathlarger{\mathbbl{\Delta}}}}

\usepackage[backend=biber,style=alphabetic]{biblatex}
\makeatletter
\newtheorem*{rep@theorem}{\rep@title}
\newcommand{\newreptheorem}[2]{%
\newenvironment{rep#1}[1]{%
 \def\rep@title{#2 \ref{##1}}%
 \begin{rep@theorem}}%
 {\end{rep@theorem}}}
\makeatother

\newreptheorem{theorem}{Theorem}

\newtheorem{theorem}{Theorem}[section]
\newtheorem{lemma}[theorem]{Lemma}
\newtheorem{proposition}[theorem]{Proposition}
\newtheorem{corollary}[theorem]{Corollary}

\newtheorem{construction}[theorem]{Construction}
\newtheorem{hypothesis}[theorem]{Hypothesis}
\newtheorem{notation}[theorem]{Notation}
\theoremstyle{remark}\newtheorem{remark}[theorem]{Remark}
\newtheorem{example}[theorem]{Example}
\crefname{construction}{construction}{constructions}
\theoremstyle{definition}\newtheorem{definition}[theorem]{Definition}
\theoremstyle{definition}

\numberwithin{equation}{theorem}

\newcommand{\Spf}{\mathrm{Spf}}

\newcommand{\FF}{\mathbb{F}}

\newcommand{\ZZ}{\mathbb{Z}}
\newcommand{\QQ}{\mathbb{Q}}

\newcommand{\CC}{\mathbb{C}}

\newcommand{\fA}{\mathbf{A}}

\newcommand{\ELn}{\mathcal{E}_{\Prism,\Lambda_n}}
\newcommand{\hEL}{\widehat{\mathcal{E}}_{\Prism,\Lambda}}

\newcommand{\fX}{\mathfrak{X}}

\author{Marvin Schneider}
\title{A note on étale $(\varphi,\Gamma)$-modules in families}

\date{\vspace{-5ex}}
\date{\today} 

\setlist[enumerate]{itemsep=1pt,parsep=1pt,before={\parskip=1pt}}

\begin{document}
\maketitle

\begin{abstract}
Let $\Lambda$ be a complete noetherian local ring with finite residue field of characteristic $p$ and $K/\mathbb{Q}_p$ a $p$-adic field.
 We show that, by deformation of the structure sheaf on the (transversal) prismatic site of a bounded $p$-adic formal scheme $\mathfrak{X}$, 
 the category of prismatic $(\Lambda,F)$-crystals on $\mathfrak{X}$ is equivalent to $\Lambda$-\'etale local systems on the generic adic fiber 
 of $\mathfrak{X}$ and that the cohomology of $(\Lambda,F)$-crystals recovers the pro-\'etale cohomology of the corresponding local systems. 
 The proof follows the strategy used in \cite{bhatt2023prismatic} and \cite{marks2023prismatic}. From this we construct an isomorphism between
Iwasawa cohomology of a $p$-adic Lie extension of $K$ and prismatic cohomology. Following \cite{wu2021galois}, we then reprove Dee's classical 
result \cite{article} on the equivalence between families of Galois representations and \'etale $(\varphi,\Gamma)$-modules.
\end{abstract}
{
  \hypersetup{linkcolor=BrickRed, linktoc=page}
  \tableofcontents
}
\clearpage

\reversemarginpar
\section{Prismatic \texorpdfstring{$(\Lambda,F)$}{(Λ, F)}-crystals}\label{main-section1}
Let $p$ be a prime. In this section we construct the main objects of interrest for this note, which are Frobenius fixed vector bundles on the ringed site $(\fX_{\Prism^\circ},\hEL)$ (cf. Definition \ref{definition-(lambda,f)-crystals}).
\begin{hypothesis}\label{hypothesis-1}
    Throughout §\ref{main-section1} we fix an inverse system $(\Lambda_n)_{n\ge 1}$ of $p$-nilpotent $\ZZ_p$-algebras, such that for each $n\ge 1$
    \begin{enumerate}[label=(\roman*)]
        \item $\Lambda_n$ is finite over $\ZZ_p$,
        \item the transition maps $\Lambda_{n+1}\to \Lambda_n$ are surjective with (locally) nilpotent kernel.
    \end{enumerate}
    We regard each $\Lambda_n$ as a topological $\ZZ_p$-algebra with the discrete topology and define $\Lambda:=\varprojlim_n \Lambda_n$,
    which we regard as a topological $\ZZ_p$-algebra carrying the inverse limit topology.
\end{hypothesis}
Let $\fX$ be a bounded\footnote{Following \cite{bhatt2023prismatic}, we call a $p$-adic formal scheme $\fX$ \emph{bounded}
    if the structure sheaf $\mathcal{O}_{\fX}$ has bounded $p$-torsion; that is if $\fX$ is Zariski locally of the form $\Spf(A)$
    with $A[p^\infty]$ having bounded exponent.} $p$-adic formal scheme and let $\fX_{\Prism^\circ}$ denote the subcategory of the
absolute prismatic site $\fX_\Prism$ of $\fX$, consisting only of transversal prisms. Recall that a bounded prism $(A,I)$ is said
to be \emph{transversal} if $A/I$ is $p$-torsion free. This automatically implies that $A/I^n$ for $n\ge 1$ and $A$ are $p$-torsion
free as well (cf. \cite[Remark 2.1.7.]{bhatt2022absolute}).
\begin{lemma}[{\cite[Lemma 2.2]{du2022new}}]
    $\fX_{\Prism^\circ}$ with the faithfully flat topology is a site.
\end{lemma}
\begin{construction}\label{construction-of-presheaves}
    For a bounded prism $(A,I)$ define
    \[
        \ELn(A,I):=(A\otimes_{\ZZ_p}\Lambda_n)^\wedge_I[\tfrac{1}{I}],
    \]
    where $(-)^\wedge_{I}$ denotes classical $I$-adic completion. The association $(A,I) \mapsto \ELn(A,I)$ defines a presheaf
    $\ELn$ on $\mathfrak{X}_\Prism$. Since $\Lambda_n$ is $p$-nilpotent and $\varphi(I)\equiv I^p\mod p$ for any prism $(A,I)$,
    the $\delta$-ring structure on prisms induces an endomorphism $\varphi_\Prism: \mathcal{E}_{\Prism,\Lambda_n}\to \mathcal{E}_{\Prism,\Lambda_n}$.
    Taking limits we define another presheaf on $\mathfrak{X}_\Prism$ by
    \[
        \hEL:=\varprojlim_n  \ELn,
    \]
    which comes with an endomorphism $\varphi_\Prism:\hEL\to \hEL$.
\end{construction}
In what follows, we will show that $\widehat{\mathcal{E}}_{\Prism,\Lambda}$ is in fact a sheaf
(at least after restriction to the site $\fX_{\Prism^\circ}$).
\begin{lemma}[{\cite[Lemma 2.2.]{steingart2022iwasawa}}]\label{flat-inverse-system-commute-with-base-change-fp}
    Let $R$ be a commutative ring and $(A_n)_{n\ge 1}$ an inverse system of flat $R$-modules with surjective transition maps.
    For any finitely presented $R$-module $N$, we have
    \[
        (\varprojlim_n A_n)\otimes_R N\cong \varprojlim_n (A_n\otimes_R N).
    \]
\end{lemma}
\begin{lemma}\label{covering-I-completely-flat}
    For a covering of transversal prisms $(A,I)\to (B,IB)$ the induced map
    \[
        (A\otimes_{\ZZ_p}\Lambda_n)^\wedge_I \to (B\otimes_{\ZZ_p}\Lambda_n)^\wedge_I
    \]
    is $I$-completely faithfully flat.
\end{lemma}
\begin{proof}
    Any bounded prism $(A,I)$ is classical $(p,I)$-complete, hence classical $I$-complete. Therefore, by
    Lemma \ref{flat-inverse-system-commute-with-base-change-fp}, we have $(A\otimes_{\ZZ_p}\Lambda_n)^\wedge_I=A\otimes_{\ZZ_p}\Lambda_n$
    for any transversal prism $(A,I)$. Moreover, by flatness of $A$ over $\ZZ_p$, $A\otimes_{\ZZ_p}^{\mathbb{L}}\Lambda_n\cong A\otimes_{\ZZ_p}\Lambda_n$.
    To prove the lemma we have to show that $(B\otimes_{\ZZ_p}\Lambda_n)\otimes^{\mathbb{L}}_{A\otimes_{\ZZ_p}\Lambda_n}(A/I\otimes_{\ZZ_p}\Lambda_n)$
    is discrete and faithfully flat as an $A/I\otimes_{\ZZ_p}\Lambda_n$-module. By $p$-nilpotentness of $\Lambda_n$ we may assume that
    $\Lambda_n$ is a $\ZZ/p^n$-algebra. We then have
    \begin{align*}
        (B/p^n\otimes_{\ZZ/p^n}\Lambda_n)\otimes^{\mathbb{L}}_{A\otimes_{\ZZ/p^n}\Lambda_n}(A/(p^n,I)\otimes_{\ZZ/p^n}\Lambda_n) 
                                                                                                                                 & \cong (B/p^n\otimes^{\mathbb{L}}_A A/(p^n,I))\otimes^\mathbb{L}_{\ZZ/p^n}\Lambda_n.
    \end{align*}
    Here we made multiple uses of \cite[\href{https://stacks.math.columbia.edu/tag/08YU}{Tag 08YU}]{stacks-project}.  
    Since $(A,I)\to (B,IB)$ is $(p,I)$-completely faithfully flat, $A/p^n \to B/p^n$ is  $I$-completely faithfully flat, whence
    \[
        B/p^n\otimes^{\mathbb{L}}_{A/p^n} A/(p^n,I)= B/p^n\otimes_A A/(p^n,I)=B/(p^n,I)
    \]
    is discrete and faithfully flat as an $A/(p^n,I)$-module. As $(B,IB)$ is transversal prism by assumption, $B/(p^n,I)\otimes^\mathbb{L}_{\ZZ/p^n}\Lambda_n= B/(p^n,I)\otimes_{\ZZ/p^n}\Lambda_n$ is concentrated in degree 0 where it is
    faithfully flat as an $A/(p^n,I)\otimes_{\ZZ/p^n}\Lambda_n$-module, since faithful flatness is stable under arbitrary base change.
\end{proof}
\begin{lemma}\label{presheaf-compatible-with-fiber-product}
    Let $(A,I)\to (B,IB)$ be a covering of transversal prisms. Then
    \begin{enumerate}[label=(\roman*)]
        \item $((B\otimes_{\ZZ_p}\Lambda_n)^\wedge_{I}  \otimes_{(A\otimes_{\ZZ_p}\Lambda_n)^\wedge_{I}}(B\otimes_{\ZZ_p}\Lambda_n)^\wedge_{I})^\wedge_I \cong ((B\otimes_A B)^\wedge_{(p,I)}\otimes_{\ZZ_p}\Lambda_n)^\wedge_I$.
        \item the derived $I$-completion of $(B\otimes_{\ZZ_p}\Lambda_n)^\wedge_{I}  \otimes^{\mathbb{L}}_{(A\otimes_{\ZZ_p}\Lambda_n)^\wedge_{I}}(B\otimes_{\ZZ_p}\Lambda_n)^\wedge_{I}$
              is concentrated in degree $0$ where it coincides with $((B\otimes_A B)^\wedge_{(p,I)}\otimes_{\ZZ_p}\Lambda_n)^\wedge_I$.
    \end{enumerate}
\end{lemma}
\begin{proof}
    By $p$-nilpotentness of $\Lambda_n$ we may assume that $\Lambda_n$ is a $\ZZ/p^n$-algebra. Using lemma
    \ref{flat-inverse-system-commute-with-base-change-fp} we have
    \begin{align*}
        ((B\otimes_A B)^\wedge_{(p,I)}\otimes_{\ZZ_p}\Lambda_n)^\wedge_I & \cong ((B/p^n\otimes_{A/p^n} B/p^n)^\wedge_{I}\otimes_{\ZZ/p^n}\Lambda_n)^{\wedge}_I                                                            \\
                                                                         & \cong ((B/p^n\otimes_{A/p^n} B/p^n)\otimes_{\ZZ/p^n}\Lambda_n)^\wedge_I                                                                         \\
                                                                         & \cong((B\otimes_{\ZZ_p}\Lambda_n)^\wedge_{I}  \otimes_{(A\otimes_{\ZZ_p}\Lambda_n)^\wedge_{I}}(B\otimes_{\ZZ_p}\Lambda_n)^\wedge_{I})^\wedge_I.
    \end{align*}
    This shows (i). For (ii) we know by the proof of \cite[Proposition 3.2.]{wu2021galois}, that the derived $(p,I)$-completion $C$ of
    $B\otimes^{\mathbb{L}}_AB$ is concentrated in degree $0$ where it coincides with $(B\otimes_A B)^\wedge_{(p,I)}$. Since $(C,IC)$ is again a
    transversal prism, we then see that $(B\otimes^{\mathbb{L}}_{A} B)^\wedge_{(p,I)}\otimes^{\mathbb{L}}_{\ZZ_p}\Lambda_n$ is concentrated in
    degree $0$ where it coincides with $(B\otimes_{A} B)^\wedge_{(p,I)}\otimes_{\ZZ_p}\Lambda_n$. As in the proof of
    Lemma \ref{covering-I-completely-flat}, we may and do replace $(-\otimes_{\ZZ_p}\Lambda_n)^\wedge_I$ with
    $(-\otimes_{\ZZ_p}^{\mathbb{L}}\Lambda_n)^\wedge_I$, where $(-)^{\wedge}_I$ now denotes derived $I$-completion. We then have
    \begin{align*}
        ((B\otimes^{\mathbb{L}}_{A} B)^\wedge_{(p,I)}\otimes^{\mathbb{L}}_{\ZZ_p}\Lambda_n)^\wedge_I & \cong  ((B/p^n\otimes^{\mathbb{L}}_{A/p^n} B/p^n)^\wedge_{I}\otimes^{\mathbb{L}}_{\ZZ/p^n}\Lambda_n)^\wedge_I                                                                          \\
                                                                                                     & \cong  (B/p^n\otimes^{\mathbb{L}}_{A/p^n} B/p^n\otimes^{\mathbb{L}}_{\ZZ/p^n}\Lambda_n)^\wedge_I                                                                                       \\
                                                                                                     & \cong ((B\otimes^{\mathbb{L}}_{\ZZ_p}\Lambda_n)^\wedge_{I}  \otimes^{\mathbb{L}}_{(A\otimes_{\ZZ_p}\Lambda_n)^\wedge_{I}}(B\otimes^{\mathbb{L}}_{\ZZ_p}\Lambda_n)^\wedge_{I})^\wedge_I
    \end{align*}
    which shows (ii).
\end{proof}
\begin{proposition}\label{hEl-sheaf}
    $\hEL$ is a sheaf on $\fX_{\Prism^\circ}$ with vanishing higher chomology on any $(A,I)\in \fX_{\Prism^\circ}$.
\end{proposition}
\begin{proof}
    We argue as in the proof of \cite[Corollary 3.12]{bhatt2022prisms}. Let $(A,I)\to (B,IB)$ be a covering of transversal prisms and
    let $(B^\bullet,IB^\bullet)$ denote its Čech nerve in $\fX_{\Prism^\circ}$. By Lemma \ref{presheaf-compatible-with-fiber-product} (ii),
    we then have that $(B^\bullet\otimes_{\ZZ_p}\Lambda_n)^\wedge_I$ is the (derived) $I$-completed Čech nerve of
    $(A\otimes_{\ZZ_p}\Lambda_n)^\wedge_I\to (B\otimes_{\ZZ_p}\Lambda_n)^\wedge_I$ considered as a cosimplicial object in
    $D((A\otimes_{\ZZ_p}\Lambda_n)^\wedge_I)$. By Lemma \ref{covering-I-completely-flat} and $I$-complete faithfully flat descent, we see
    that $(A\otimes_{\ZZ_p}\Lambda_n)^\wedge_I\to (B^\bullet\otimes_{\ZZ_p}\Lambda_n)^\wedge_I$ is a limit diagram in
    $D((A\otimes_{\ZZ_p}\Lambda_n)^\wedge_I)$. Inverting $I$ and taking limits over $n$ (we actually take the $R\lim$ here, but by
    surjectivity of the transition maps, $R^1\lim=0$) we have that $\hEL(A,I)\to \hEL(B^\bullet,IB^\bullet)$ is a limit diagram
    in $D(\hEL(A,I))$. This shows that $\hEL$ is a sheaf, as well as the vanishing of higher Čech cohomology.
\end{proof}
\begin{notation}
    Let $\fX_{\Prism}^{\mathrm{perf}}$ denote the full subcategory of $\fX_\Prism$ consisting of perfect prisms and
    $\fX_{\Prism^\circ}^{\mathrm{perf}}$ the subcategory consisting of transversal perfect prisms.
\end{notation}
\begin{remark}
    By Proposition \ref{hEl-sheaf} $\hEL$ is a sheaf on $\fX_{\Prism^\circ}^{\text{perf}}$.
\end{remark}
In a next step we show that $\hEL$ is a sheaf on $\fX_{\Prism}^{\text{perf}}$, which will be used in §\ref{main-section2} to relate
(the yet to be defined) category of prismatic $(\Lambda,F)$-crystals on $\fX_{\Prism^\circ}$ to local systems. To this end, we remind
of the correspondence between perfect prisms and perfectoid rings.
\begin{theorem}[{\cite[Theorem 3.10]{bhatt2022prisms}}]\label{perfectprism-perfectoidrings}
    The functor $(A,I)\mapsto A/I$ induces an equivalence between the category of perfect prisms and the category of perfectoid rings.
    A quasi-inverse functor is given by $R\mapsto (A_{\mathrm{inf}}(R),\mathrm{ker}(\theta))$.
\end{theorem}
\begin{corollary}\label{perf-sheaf}
    $\hEL$ is a sheaf on $\fX_{\Prism}^{\mathrm{perf}}$.
\end{corollary}
\begin{proof}
    Let $(A,I)$ denote a perfect prism and let $R=A/I$ be the corresponding perfectoid ring. Since perfectoid rings are reduced,
    $R$ is either $p$-torsion free or a perfect $\FF_p$-algebra. Therefore, $(A,I)$ is either a transversal prism or crystalline
    (i.e. $I=(p)$). For crystalline prisms $(A,(p))$, $\hEL(A,(p))=0$ by $p$-nilpotentness of $\Lambda_n$. Let $(A,I)\to (B,J)$
    be a faithfully flat map of perfect prism. By \cite[Remark 2.4.4. ]{bhatt2022absolute}, we have that $(A,I)$ transversal if
    and only if $(B,J)$ is transversal, hence $(A,I)$ is crystalline if and only if $(B,J)$ is crystalline. The claim then follows,
    since $\hEL$ is a sheaf on $\fX_{\Prism^\circ}^{\text{perf}}$ and $=0$ on $\fX_{\Prism}^{\text{perf}}$ outside
    $\fX_{\Prism^\circ}^{\text{perf}}$.
\end{proof}
We will now investigate the behavior of vector bundles on the ringed site $(\fX_{\Prism^\circ},\hEL)$. This will ultimately
lead to the definition of prismatic $(\Lambda,F)$-crystals.
\begin{notation}[{\cite[Notation 2.1]{bhatt2023prismatic}}]
    Let $(\mathcal{X},\mathcal{O})$ be a ringed topos and $\ast$ the final object of $\mathcal{X}$.
    \begin{enumerate}[label=(\roman*)]
        \item For a commutative ring $A$ we write $\mathrm{Vect}(A)$ for the category of finite projective $A$-modules. 
        \item A vector bundle on $(\mathcal{X},\mathcal{O})$ is an $\mathcal{O}$-module $E$ such that there exists a cover
              $\{U_i\to \ast\}$ and finite projective $\mathcal{O}(U_i)$-modules $P_i$ such that
              $E_{U_i}\cong P_i\otimes_{\mathcal{O}(U_i)}\mathcal{O}_{U_i}$. We denote the category of vector bundles on
              $(\mathcal{X},\mathcal{O})$ by $\mathrm{Vect}(\mathcal{X},\mathcal{O})$.
    \end{enumerate}
\end{notation}
\begin{proposition}\label{vector-bundle-on-hEL}
    There is an equivalence of categories \[
        \mathrm{Vect}(\fX_{\Prism^\circ}, \hEL) \cong \lim_{(A,I)\in \fX_{\Prism^\circ}}\mathrm{Vect}(\hEL(A,I)).
    \]
\end{proposition}
\begin{proof}
    Identifying $\mathrm{Vect}(\fX_{\Prism^\circ}, \hEL)$ with the global sections of the stackification of the fibered category
    $(A,I)\to \mathrm{Vect}(\hEL(A,I))$, it suffices to show that this assignment is already a sheaf. The kernel of
    $\mathcal{E}_{\Prism,\Lambda_{n+1}}(A,I)\to \mathcal{E}_{\Prism,\Lambda_n}(A,I)$ is generated by the image of
    $\ker (\Lambda_{n+1}\to \Lambda_n)$ under the map $\Lambda_{n+1}\to (A\otimes_{\ZZ_p}\Lambda_{n+1})^\wedge_{I}[1/I]=A[1/I]\otimes_{\ZZ_p}\Lambda_{n+1}$.
    Since the latter kernel is nilpotent and finitely generated, the former kernel is as well. Combined with the surjectivity of the
    transition maps this induces an equivalence of categories
    \[
        \mathrm{Vect}(\hEL(A,I)) \cong 2\text{ -}\varprojlim_n \mathrm{Vect}(\ELn(A,I))
    \]
    by \cite[\href{https://stacks.math.columbia.edu/tag/0D4B}{Tag 0D4B}]{stacks-project}. As the 2-limit of stacks is again a stack,
    it suffices to show that for any $n \ge 1$ the fibered category \[
        (A,I)\to \mathrm{Vect}(\mathcal{E}_{\Prism,\Lambda_n}(A,I))
    \]
    is a stack. This follows from Lemma \ref{covering-I-completely-flat}, Lemma \ref{presheaf-compatible-with-fiber-product} and
    the descent result \cite[Theorem 5.8]{mathew2021faithfully}.
\end{proof}
\begin{remark}
    Arguing as in Corollary \ref{perf-sheaf}, one shows that Proposition \ref{vector-bundle-on-hEL} holds when $\fX_{\Prism^\circ}$ is
    replaced by $\fX_{\Prism}^{\text{perf}}$.
\end{remark}
\begin{definition}
    Let $R$ be a commutative ring and $\varphi:R\to R$ an endomorphism. Pullback along $\varphi$ induces a functor
    $\mathrm{Vect}(R)\overset{\varphi^\ast}{\to}\mathrm{Vect}(R)$. We define the category of étale $\varphi$-modules over $R$ as the
    (2-)fiber product of $\mathrm{Vect}(R)\overset{\varphi^\ast}{\to}\mathrm{Vect}(R)\overset{\operatorname{id}}{\leftarrow}\mathrm{Vect}(R)$,
    which we denote by $\mathrm{Vect}(R)^{\varphi=1}$.
\end{definition}
\begin{remark}
    Objects in $\mathrm{Vect}(R)^{\varphi=1}$ are finite projective $R$-modules $M$ together with an $R$-module isomorphism
    $f:\varphi^\ast M\cong M$. By defining $\varphi_M(-):=f(1\otimes -)$, one sees that this is just the category of finite
    projective $R$-modules $M$ together with a $\varphi$-semi linear endomorphism $\varphi_M:M\to M$ for which the linearization
    $\varphi_M^{\text{lin}}$ is an isomorphism of $R$-modules. This coincides with the usual definition of étale $\varphi$-modules.
\end{remark}
\begin{definition}\label{definition-(lambda,f)-crystals}
    The category of prismatic $(\Lambda,F)$-crystals on $\fX_{\Prism^\circ}$ is defined as \[
        \mathrm{Vect}(\fX_{\Prism^\circ},\hEL)^{\varphi=1}:= \lim_{(A,I)\in \fX_{\Prism^\circ}}\mathrm{Vect}(\hEL(A,I))^{\varphi=1}.
    \]
    We similarly define the categories $\mathrm{Vect}(\fX_{\Prism^\circ}^{\mathrm{perf}},\hEL)^{\varphi=1}$
    and $\mathrm{Vect}(\fX_{\Prism}^{\mathrm{perf}},\hEL)^{\varphi=1}$.
\end{definition}
\begin{remark}
    Thanks to Proposition \ref{vector-bundle-on-hEL}, the category of $(\Lambda,F)$-crystals on $\fX_{\Prism^\circ}$ identifies
    with the category of pairs $(E,f)$, where $E\in \mathrm{Vect}(\fX_{\Prism^\circ},\hEL)$ and $f$ is an isomorphism $f:\varphi^\ast E\cong E$
    (here: $\varphi=\varphi_\Prism$ as defined in Construction \ref{construction-of-presheaves}). One has similar descriptions for the categories
    $\mathrm{Vect}(\fX_{\Prism^\circ}^{\mathrm{perf}},\hEL)^{\varphi=1}$ and $\mathrm{Vect}(\fX_{\Prism}^{\mathrm{perf}},\hEL)^{\varphi=1}$.
\end{remark}
\begin{example}
    Let $\fX$ be a $p$-torsion free quasi-syntomic $p$-adic formal scheme. For $\Lambda=\ZZ_p$ and $\Lambda_n=\ZZ/p^n$,
    the category of $(\Lambda,F)$-crystals on $\fX_{\Prism^\circ}$ is equivalent to the category of Laurent $F$-crystals on $\fX_{\Prism}$
    in the sense of \cite[Definition 3.2]{bhatt2023prismatic}. Indeed, it is clear that in this case the structure sheaf $\hEL$ is just
    $\mathcal{O}_\Prism[1/\mathcal{I}_\Prism]^\wedge_p$, and by \cite[Lemma 2.3.]{du2022new} pullback along $\fX_{\Prism^\circ}\to \fX_{\Prism}$
    induces an equivalence of categories
    \[
        \mathrm{Vect}(\fX_{\Prism},\mathcal{O}_\Prism[1/\mathcal{I}_\Prism]^\wedge_p)^{\varphi=1}\cong \mathrm{Vect}(\fX_{\Prism^\circ},\mathcal{O}_\Prism[1/\mathcal{I}_\Prism]^\wedge_p)^{\varphi=1}.
    \]
    By \cite[Corollary 3.8]{bhatt2023prismatic}, the latter category is equivalent to the category of $\ZZ_p$-local systems on the generic fiber of $\fX$. In §\ref{main-section2}, we will extend this result to more general coefficients $\Lambda$. 
\end{example}
\section{Local systems and prismatic \texorpdfstring{$(\Lambda,F)$}{(Λ, F)}-crystals}\label{main-section2}
The object of this section is to study $(\Lambda,F)$-crystals under additional finiteness assumptions on $\Lambda$ and relate them and their cohomology
to (pro-étale) local systems on the generic adic fiber $\fX_\eta$ of $\fX$. As an application we construct a functor $\Prism_{\text{Iw}}:\mathrm{Rep}_{\ZZ_p}(G_K)\to \mathrm{Vect}((\mathcal{O}_K)_{\Prism^\circ},\hEL)^{\varphi=1}$
which computes the Iwasawa cohomology of a $p$-adic Lie extension of a $p$-adic field $K/\QQ_p$.
\begin{hypothesis}\label{hypothesis-2}
    Throughout §\ref{main-section2} we fix a complete noetherian local ring $(\Lambda,\mathfrak{m})$ with finite residue field $k_\Lambda$ 
    of characteristic $p$. Let $\Lambda_n:=\Lambda/\mathfrak{m}^n$ for $n\ge 1$. Then $\Lambda=\varprojlim_n \Lambda_n$, which puts 
    us in the setting of Hypothesis \ref{hypothesis-1}.
\end{hypothesis}
\begin{remark}
In the situation of Hypothesis \ref{hypothesis-2} the sheaf $\hEL$ is given by $(A,I)\mapsto A[\tfrac{1}{I}]\widehat{\otimes}_{\ZZ_p}\Lambda$, 
where $\widehat{\otimes}_{\ZZ_p}$ denotes the classical $\mathfrak{m}$-adic completion of the usual tensor product.
\end{remark}
\begin{notation}[{\cite[Notation 3.1.]{bhatt2023prismatic}}]
For a bounded $p$-adic formal scheme $\mathfrak{X}$ we shall write $\mathfrak{X}_\eta$ for its adic generic fiber which we regard as a 
locally spatial diamond (cf. \cite[§39]{scholze2022etale}). For a scheme, formal scheme, or adic space $X$ we write $X_{\text{ét}}$ 
for the étale site of $X$. For $n\ge 1$, we shall write $\mathrm{Loc}_{\Lambda_n}(X)$ for the category of étale $\Lambda_n$-local systems
 on $X_{\text{ét}}$, i.e. sheaves $\mathcal{L}$ of flat $\Lambda_n$-modules, which locally on $X_{\text{ét}}$ are constant sheaves
associated to a finitely generated $\Lambda_n$-module. Recall that one has the notion of lisse $\Lambda$-local systems on $X_{\text{ét}}$, 
i.e. inverse systems $(\mathcal{L}_n)_{n\ge 1}$, such that each $\mathcal{L}_n$ is an étale $\Lambda_n$-local system on $X_{\text{ét}}$,
and such that the transition maps $\mathcal{L}_{n+1}\to \mathcal{L}_n$ induce an isomorphism $\Lambda_n\otimes_{\Lambda_{n+1}} \mathcal{L}_{n+1} \overset{\sim}{\to} \mathcal{L}_{n}$ (cf. \cite[definition 8.1]{scholze2012padic}).
We denote this category by $\mathrm{Loc}_\Lambda(X)$. Note that by (the proof of) \cite[Proposition 8.2.]{scholze2022etale} 
$\mathrm{Loc}_\Lambda(X)$ identifies with the category of lisse $\hat{\Lambda}$-sheaves on the pro-étale site $X_{\text{proét}}$ in the 
sense of \cite[Definition 8.1.]{scholze2022etale}.
If $X=\mathrm{Spec}(A)$ is affine we simply write $\mathrm{Loc}_\Lambda(A)$ for $\mathrm{Loc}_\Lambda(\mathrm{Spec}(A))$
and similarly for the other categories.
\end{notation}
The main result of this section is the following analogue of \cite[Theorem 5.15. (i)]{marks2023prismatic}.
\begin{theorem}\label{mainthm-1}
    Let $\mathfrak{X}$ be a bounded $p$-adic formal scheme.
    There are equivalences of categories
        \begin{align*}
            \mathrm{Vect}(\fX_{\Prism^\circ},\hEL)^{\varphi=1} \cong \mathrm{Vect}(\fX^\mathrm{perf}_{\Prism},\hEL)^{\varphi=1} \cong \mathrm{Loc}_\Lambda(\mathfrak{X}_\eta).
        \end{align*}
\end{theorem}
The proof of Theorem \ref{mainthm-1} is formally the same as the one of \cite[Theorem 5.15. (i)]{marks2023prismatic}. Nevertheless, we 
give an overview of the general idea and show how to adapt the main arguments that went into the proof of \emph{loc. cit.} to our setting. We first consider the case $\Lambda=k_\Lambda$ and then use the usual dévissage argument and taking limits to prove the general case.

\begin{lemma}\label{perfection-same-phi-modules}
    Let $(A,I)$ be a transversal prism with perfection $(A_{\mathrm{perf}},IA_\mathrm{perf})$. Base change induces an equivalance of categories
    \[
    \mathrm{Vect}(\hEL(A,I))^{\varphi=1}\cong \mathrm{Vect}(\hEL(A_\mathrm{perf},IA_\mathrm{perf}))^{\varphi=1}.   
    \]
\end{lemma}
\begin{proof}
    By $\mathfrak{m}$-adic completeness and arguing with dévissage modulo powers of $\mathfrak{m}$,  we reduce to the case $\Lambda=k_\Lambda$. 
    That is, we have to show that base change induces an equivalence
    \[
        \mathrm{Vect}(A[\tfrac{1}{I}]/p\otimes_{\FF_p}k_\Lambda)^{\varphi=1}\cong \mathrm{Vect}(A_{\mathrm{perf}}[\tfrac{1}{I}]/p\otimes_{\FF_p}k_\Lambda)^{\varphi=1}. 
    \]
    Let $S:=A[\tfrac{1}{I}]/p$ and $T:=A_{\mathrm{perf}}[\tfrac{1}{I}]/p$. As in the proof of \cite[Proposition 5.4.]{marks2023prismatic} one checks 
    that base change induces an equivalence of étale sites $S_{\text{ét}}\cong T_{\text{ét}}$. Since $k_\Lambda/\FF_p$ is finite étale, 
    we see, by localization of both sites, that the étale sites of $S\otimes_{\FF_p} k_\Lambda$ and $T\otimes_{\FF_p}k_\Lambda$ are equivalent as well.
    Let $M$ denote a finite projective $T\otimes_{\FF_p}k_\Lambda$-module with an isomorphism $(\varphi\otimes \operatorname{id})^\ast M\cong M$ of $T\otimes_{\FF_p}k_\Lambda$-modules. 
    Then $(\varphi^n\otimes \operatorname{id})^\ast M \cong M$ for any $n\ge 1$. For a suitable choice of $n$, by finiteness of $k_\Lambda/\FF_p$, 
    $\varphi^n\otimes \operatorname{id}$ acts on $S\otimes_{\FF_p}k_\Lambda$ (and $T\otimes_{\FF_p}k_\Lambda$) via a power of the Frobenius. The result then follows from \cite[Proposition 4.1.1]{10.1007/978-3-540-37802-0_3} by keeping track of the isomorphism $\varphi^\ast M\cong M$.
    \end{proof}
\begin{lemma}\label{phi-modules-equiv-etale-site}
    Let $A$ be a perfect $\FF_p$-algebra such that $W(A)/p^n$ is flat over $\ZZ/p^n$. Let the usual Frobenius lift on $W(A)/p^n$ act $\Lambda_n$-linearly 
    on $W(A)/p^n\otimes_{\ZZ/p^n}\Lambda_n$. There is an equivalence 
    \[
    \mathrm{Vect}(W(A)/p^n\otimes_{\ZZ/p^n}\Lambda_n)^{\varphi=1}\cong\mathrm{Loc}_{\Lambda_n}(A).
    \]
\end{lemma}
\begin{proof}
    The case $\Lambda_n=\ZZ/p^n$ is \cite[Proposition 3.2.7.]{kedlaya2015relative}. In the general case we can argue as in 
    \cite[Proposition 3.2.]{emerton2001unit}.
\end{proof}
\begin{proposition}\label{phi-mod-equiv-lisse}
    Let $(A,I)$ be a transversal prism with perfection $(A_{\mathrm{perf}},IA_\mathrm{perf})$. There are equivalences of categories
    \[
    \mathrm{Vect}(\hEL(A,I))^{\varphi=1}\cong \mathrm{Vect}(\hEL(A_{\mathrm{perf}},IA_\mathrm{perf}))^{\varphi=1}\cong \mathrm{Loc}_\Lambda(A_{\mathrm{perf}}[\tfrac{1}{I}]/p). 
    \]
\end{proposition}
\begin{proof}
    The first equivalence is Lemma \ref{perfection-same-phi-modules}. The second equivalence follows from Lemma \ref{phi-modules-equiv-etale-site} and taking limits.
\end{proof}
\begin{remark}\label{perctiond-fina-object}
    Let $R$ be a perfectoid ring. By Theorem \ref{perfectprism-perfectoidrings} the corresponding perfect prism is given by 
    $(A_{\mathrm{inf}}(R),\mathrm{ker}(\theta))$ and is the final object of $(R)_\Prism^{\mathrm{perf}}$ by \cite[Lemma 4.8.]{bhatt2022prisms}.
\end{remark}
\begin{corollary}\label{thm1-for-perfectoid-rings}
    Let $R$ be a perfectoid ring. There are equivalences of categories
    \[
    \mathrm{Vect}(R_{\Prism^\circ},\hEL)^{\varphi=1}\cong \mathrm{Vect}(\hEL(A_{\mathrm{inf}}(R),\ker\theta))^{\varphi=1} \cong \mathrm{Loc}_\Lambda(R[\tfrac{1}{p}]).
    \]
    Moreover, for any $\mathcal{M}\in\mathrm{Vect}(R_{\Prism^\circ},\hEL)^{\varphi=1}$ with corresponding local system $\mathcal{L}$, there is a quasi-isomorphism\[
        R\Gamma(R_{\Prism^\circ},\mathcal{M})^{\varphi=1}\cong R\Gamma(R[\tfrac{1}{p}]_{\text{proét}},\mathcal{L}).
    \]
\end{corollary}
\begin{proof}
This is proven analogouly to \cite[Proposition 5.13.]{marks2023prismatic} and \cite[Corollary 5.14.]{marks2023prismatic}.
\end{proof}
\begin{proof}[Proof (of Theorem \ref{mainthm-1}).]
    Let $X$ denote an analytic adic space over $\mathrm{Spa}(\QQ_p,\ZZ_p)$. By \cite[Lemma 15.6.]{scholze2022etale}, $X$ can be viewed as a 
    locally spatial diamond.  By \cite[Proposition 3.7.]{mann2020local}, \cite[Proposition 3.9.]{mann2020local} and discreteness of $\Lambda_n$, 
    pullback along $X_v \to X_{\text{ét}}$ induces an equivalence $\mathrm{Loc}_{\Lambda_n}(X_{\text{ét}})  \cong \mathrm{Loc}_{\Lambda_n}(X_{v})$ 
    where $X_{v}$ denotes the v-site of $X$. Taking limits on both sides it then follows
    \begin{align}
        \mathrm{Loc}_{\Lambda}(X_{\text{ét}})  \cong \lim_n \mathrm{Loc}_{\Lambda_n}(X_{v}) \cong \mathrm{Loc}_{\Lambda}(X_{v})\label{etale-v-equiv}
    \end{align}
    where the second equivalence is \cite[Proposition 3.5.]{mann2020local}. We have
    \begin{align*}
        \mathrm{Vect}(\fX_{\Prism^\circ},\hEL)^{\varphi=1} \cong \lim_{(A,I)\in \fX^{\mathrm{perf}}_{\Prism}}\mathrm{Vect}(\hEL(A,I))^{\varphi=1} \cong \lim_{\substack{\mathrm{Spf}(A)\to \fX \\ A \text{ perfectoid}}} \mathrm{Loc}_{\Lambda}(A[\tfrac{1}{p}]_{\text{ét}})
    \end{align*}
    by Proposition \ref{phi-mod-equiv-lisse} and Theorem \ref{perfectprism-perfectoidrings}. On the other hand,
    \begin{align*}
        \lim_{\substack{\mathrm{Spf}(A)\to \fX \\ A \text{ perfectoid}}} \mathrm{Loc}_{\Lambda}(A[\tfrac{1}{p}]_{\text{ét}})\cong  \lim_{\substack{\mathrm{Spf}(A)\to \fX \\ A \text{ perfectoid}}} \mathrm{Loc}_{\Lambda}(A[\tfrac{1}{p}]_{v})\cong\mathrm{Loc}_{\Lambda}(X_{v})\cong\mathrm{Loc}_{\Lambda}(X_{\text{ét}})
    \end{align*}
where the first and last equivalences are by \ref{etale-v-equiv} and the second equivalence by \cite[Proposition 15.4.]{scholze2022etale}.
\end{proof}
\begin{definition}
The functor $T^\Lambda_\fX: \mathrm{Vect}(\fX_{\Prism^\circ},\hEL)^{\varphi=1} \cong \mathrm{Loc}_\Lambda(\mathfrak{X}_\eta)$ in Theorem \ref{mainthm-1} is called the étale realization functor.
\end{definition}
Following the proof of \cite[Theorem 5.15. (ii)]{marks2023prismatic}, which is formally identical to the proof of Theorem \ref{mainthm-1}, we have the following étale comparison isomorphism.
\begin{theorem}\label{prismatic-proetale-comparison}
Let $\mathcal{M}\in \mathrm{Vect}(\fX_{\Prism^\circ},\hEL)^{\varphi=1}$. There is a quasi-isomorphism\[
R\Gamma(\fX_{\Prism^\circ},\mathcal{M})^{\varphi=1}\cong R\Gamma(\fX_{\eta,\text{proét}},T^\Lambda_\fX(\mathcal{M})).
\]
\end{theorem}
Let $K/\QQ_p$ be a $p$-adic field. We write $G_K$ for the absolute Galois group of $K$ and $\mathrm{Rep}_{\ZZ_p}(G_K)$ for the category of (continuous) representations of $G_K$ on finite projective $\ZZ_p$-modules. Let $K_\infty/K$ denote a $p$-adic Lie extension of $K$ with Galois group $\mathrm{Gal}(K_\infty/K)=\ZZ_p^d$ for some $d\ge 1$ and choose a tower of finite Galois extensions\[    
K=K_1\subset K_2\subset K_3 \subset \dots \subset K_n \subset K_\infty
\]
with $K_\infty=\cup_{n\ge 1}K_n$.
\begin{definition}
Let $V\in \mathrm{Rep}_{\ZZ_p}(G_K)$. The Iwasawa cohomology of $K_\infty/K$ with values in $V$ is defined as the inverse limit\[
H_{\text{Iw}}^{\bullet}(K_\infty/K,V):=\varprojlim_n H^\bullet(G_{K_n},V)
\] 
where $H^\bullet(G_{K_n},V)$ denotes continuous group cohomology and the transitions maps are given by cohomological corestriction maps.\footnote{A simple cofinality 
argument shows that $H_{\text{Iw}}^{\bullet}(K_\infty/K,V)$ does not depend on the choice of $(K_n)_{n\ge 1}$.}
\end{definition}
As an application of Theorem \ref{prismatic-proetale-comparison} we have the following comparison between prismatic cohomology and Iwasawa cohomology. 
Let $H_\varphi^{\bullet}((\mathcal{O}_K)_{\Prism^\circ},-)$ denote the cohomology of the complex $R\Gamma((\mathcal{O}_K)_{\Prism^\circ},-)^{\varphi=1}$.
\begin{theorem}
Given a Galois extension $K_\infty/K$ with Galois group $\Gamma$ isomorphic to $\ZZ_p^d$ for some $d\ge 1$, let $\Lambda:=\ZZ_p[\![\Gamma]\!]$ denote the completed group algebra over $\ZZ_p$. There exists a functor 
$\Prism_{\mathrm{Iw}}:\mathrm{Rep}_{\ZZ_p}(G_K)\to \mathrm{Vect}((\mathcal{O}_K)_{\Prism^\circ},\hEL)^{\varphi=1}$ such that 
\[
    H_{\mathrm{Iw}}^{\bullet}(K_\infty/K,V) \cong H_\varphi^{\bullet}((\mathcal{O}_K)_{\Prism^\circ},\Prism_{\mathrm{Iw}}(V)).
\]
\end{theorem}
\begin{proof}
It is well known that $\Lambda$ is isomorphic to $\ZZ_p[\![X_1,\dots,X_n]\!]$, hence we are in the situation of Hypothesis \ref{hypothesis-2}. Note that $\Lambda$ admits an action of $G_K$ via left multiplication 
via the quotient map $G_K\twoheadrightarrow \Gamma$. By (the proof of) \cite[Lemma 5.8.]{schneider2015coateswiles} we have 
\[
    H_{\text{Iw}}^{\bullet}(K_\infty/K,V)\cong H^{\bullet}(G_K,\Lambda\otimes_{\ZZ_p}V),
\]
where $G_K$ acts diagonally on $\Lambda\otimes_{\ZZ_p}V$. 
Since $\Lambda\otimes_{\ZZ_p}V \in \mathrm{Rep}_{\Lambda}(G_K)$, by Theorem \ref{mainthm-1}, there exsits a $(\Lambda,F)$-crystal $\mathcal{M}$ on $(\mathcal{O}_K)_{\Prism^\circ}$ such that\[
    H_{\text{Iw}}^{\bullet}(K_\infty/K,V)\cong H^{\bullet}(G_K,\Lambda\otimes_{\ZZ_p}V)\cong   H_\varphi^{\bullet}((\mathcal{O}_K)_{\Prism^\circ},\mathcal{M})
\]
by Theorem \ref{prismatic-proetale-comparison}. After fixing a choice of a quasi-inverse functor $G$ of $T^\Lambda_{\mathrm{Spf}(\mathcal{O}_K)}$, one can simply define $\Prism_{\text{Iw}}(V):=G(\Lambda\otimes_{\ZZ_p}V)$, which is functorial in $V$.
\end{proof}
\section{Families of \texorpdfstring{$(\varphi,\Gamma)$-modules}{(varphi,Gamma)-modules} and prismatic \texorpdfstring{$(\Lambda,F)$}{(Λ, F)}-crystals}\label{main-section3}
In the first part of this section we apply the results of \cite{wu2021galois} to Theorem \ref{mainthm-1} and give an alternative proof of 
the main theorem in \cite{article}. In the second part, we recover the quasi-isomorphism between Galois cohomology and the Herr complex proven in \cite{article}.
\begin{hypothesis}
Throughout §\ref{main-section3} we let $(\Lambda,\mathfrak{m})$ be as in Hypothesis §\ref{main-section2}. Moreover, we fix a perfect field $k$ of 
characteristic $p$, define $K':=W(k)[1/p]$ and let $K/K'$ denote a finite totally ramified extension. Let $\CC_p$ be the completion of an 
algebraic closure of $K$.
\end{hypothesis}
\begin{example} [{\cite{wu2021galois}}]
Let $(A_\mathrm{inf}(\mathcal{O}_{\CC_p}),\ker(\theta))$ denote the (transversal) perfect prism corresponding to the perfectoid ring 
$\mathcal{O}_{\CC_p}$ under the equivalence in Theorem \ref{perfectprism-perfectoidrings}. Let $(\zeta_{p^n})\subset \CC_p$ be a 
compatible system of $p$-power roots of unity and $\epsilon:=(1,\zeta_p,\zeta_{p^2},\dots)\in \mathcal{O}_{\CC_p^\flat}$.
By the theory of field of norms one can associate to the cyclotomic tower \[
K\subset K(\zeta_p) \subset K(\zeta_{p^2})\subset \dots K_\infty:=K(\zeta_{p^\infty})
\]
the field of norms $\mathbf{E}_K$, which is contained in $\hat{K}_\infty^\flat$ and admits a $\varphi$-stable Cohen ring $\mathbf{A}_K$
contained in $W(\CC_p^\flat)$. $\mathbf{A}_K$ contains the ring $W(k)[\![q-1]\!]$, with the element $q$ being mapped 
to $[\epsilon]\in W(\CC_p^\flat)$. Let $[p]_q:=\frac{q^p-1}{q-1}\in W(k)[\![q-1]\!]$ and 
$\mathbf{A}_K^+:=\mathbf{A}_K\cap A_\text{inf}(\mathcal{O}_{\CC_p})$. Then $(\mathbf{A}_K^+, (\varphi^n([p]_q)))$ is a prism for 
every $n\ge 1$ \cite[Lemma 2.5]{wu2021galois}.
\end{example}
\begin{lemma}\label{final-object}
    There exsits $n\ge 1$ such that  $(\mathbf{A}_K^+, (\varphi^n([p]_q)))\in (\mathcal{O}_K)_{\Prism^\circ}$.
\end{lemma}
\begin{proof}
    We have that $\mathbf{A}_K^+/(\varphi^n([p]_q))$ is $p$-torsion free for every $n\ge 1$ by \cite[Corollary 2.7]{wu2021galois}.
    By \cite[Proposition 2.19]{wu2021galois}, there exists $n\ge 1$ such that 
    $(\mathbf{A}_K^+, (\varphi^n([p]_q)))\in (\mathcal{O}_K)_{\Prism}$, thus $(\mathbf{A}_K^+, (\varphi^n([p]_q)))\in (\mathcal{O}_K)_{\Prism^\circ}$.
\end{proof}
\begin{construction}
Let $\mathbf{A}_{K,\Lambda}:=\hEL(\mathbf{A}_K^+, (\varphi^n([p]_q)))$. By \cite[Corollary 2.18.]{wu2021galois} the automorphism group of 
    $(\mathbf{A}_K^+, (\varphi^n([p]_q)))$ in $(\mathcal{O}_K)_\Prism$ is isomorphic to $\Gamma:=\mathrm{Gal}(K_\infty/K)$. The action of $\Gamma$ commutes with 
    $\varphi$ and extends $\Lambda$-linearly to an action on $\mathbf{A}_{K,\Lambda}$. By \cite[Lemma 2.17.]{wu2021galois} we have that the perfection of 
    $(\mathbf{A}_K^+, (\varphi^n([p]_q)))$ is given by the perfect prism $(A_\mathrm{inf}(\mathcal{O}_{\hat{K}_\infty}),\ker\theta)$ 
    corresponding to the perfectoid ring $\mathcal{O}_{\hat{K}_\infty}$. We put $\tilde{\mathbf{A}}_{K,\Lambda}:=\hEL(A_\mathrm{inf}(\mathcal{O}_{\hat{K}_\infty}),\ker\theta)$. 
    As above, this ring is equipped with a $\Lambda$-linear action of $\Gamma$ commuting with $\varphi$.
\end{construction}
\begin{definition}
    Let $R$ be a commutative ring, $\varphi:R\to R$ an endomorphism and $G\subseteq \mathrm{Aut}(R)$ a subgroup of the automorphism group of $R$ such that $g\circ \varphi= \varphi\circ g$ 
    for all $g\in G$. The category of étale $(\varphi, G)$-modules over $R$ is defined to be the subcategory of $\mathrm{Vect}(R)^{\varphi=1}$ consisting of étale $\varphi$-modules together 
    with a semi-linear action of $G$ by $R$-module automorphisms commuting with $\varphi$. Morphisms are given by those morphisms of $\mathrm{Vect}(R)^{\varphi=1}$ fixed by the action of $G$. 
    We denote the category of étale $(\varphi,G)$-modules by $\mathrm{Mod}_{\varphi,G}^\text{ét}(R)$.
\end{definition}
\begin{proposition}\label{phi-gamma-perf-equiv}
Base extension induces an equivalence of categories\[
    \mathrm{Mod}_{\varphi,\Gamma}^\text{ét}(\fA_{K,\Lambda})\cong \mathrm{Mod}_{\varphi,\Gamma}^\text{ét}(\tilde{\fA}_{K,\Lambda}).
\]
\end{proposition}
\begin{proof}
    This follows from Proposition \ref{perfection-same-phi-modules} by taking the $\Gamma$-action into account.
\end{proof}
\begin{lemma}\label{cech-nerve-ainf}
    Let $(A^\bullet, \ker\theta)$ be the Čech nerve of $(A_\mathrm{inf}(\mathcal{O}_{\hat{K}_\infty}),\ker\theta)$ in $(\mathcal{O}_K)_\Prism^{\mathrm{perf}}$. There is an isomorphism of cosimplicial rings
    \[
        \hEL(A^\bullet, \ker\theta)\cong C(\Gamma^\bullet, \tilde{\mathbf{A}}_{K,\Lambda})
    \]
    where $C(\Gamma^n,\tilde{\mathbf{A}}_{K,\Lambda})$ denotes the ring of continuous maps $\Gamma^n\to \tilde{\mathbf{A}}_{K,\Lambda}$.
\end{lemma}
\begin{proof}
    The case $\Lambda=\ZZ_p$ is \cite[Lemma 5.3.]{wu2021galois}. In general, it suffices to show that there is an isomorphism
\begin{align*}
    C(\Gamma^\bullet, \tilde{\mathbf{A}}_{K,\ZZ_p})\hat{\otimes}_{\ZZ_p}\Lambda \cong C(\Gamma^\bullet, \tilde{\mathbf{A}}_{K,\Lambda}).
\end{align*}
    Let $\ker \theta =(\alpha)$ and $\tilde{\mathbf{A}}_K^+:=A_\mathrm{inf}(\mathcal{O}_{\hat{K}_\infty})$. We have $\tilde{\mathbf{A}}_{K,\ZZ_p}=\tilde{\mathbf{A}}_K^+[\frac{1}{\alpha}]^\wedge_p$
    and by $\mathfrak{m}$-adic completeness it suffices to show that
     \[
        C(\Gamma^\bullet, \tilde{\mathbf{A}}_K^+[\frac{1}{\alpha}]/p^n)\otimes_{\ZZ/p^n}\Lambda_n \cong C(\Gamma^\bullet, \tilde{\mathbf{A}}_K^+[\frac{1}{\alpha}]/p^n \otimes_{\ZZ/p^n}\Lambda_n)
    \]
    is an isomorphism for all $n\ge 1$. By compactness of $\Gamma$ (cf. \cite[Lemma 4.3.7.]{bhatt2014proetale}) and $\alpha$-adic completeness of $\tilde{\mathbf{A}}_K^+/p^n$ it suffices to show that
    \[
        C(\Gamma^\bullet, \tilde{\mathbf{A}}_K^+/(\alpha^m,p^n))\otimes_{\ZZ/p^n}\Lambda_n \cong C(\Gamma^\bullet, \tilde{\mathbf{A}}_K^+/(\alpha^m,p^n) \otimes_{\ZZ/p^n}\Lambda_n)
    \]
    is an isomorphism for all $n,m \ge 1$. This is clear if $\Lambda_n$ is flat over $\ZZ/p^n$. In general, by finiteness of $\Lambda_n$ over $\ZZ/p^n$, 
    one can apply $C(\Gamma^s,\tilde{\mathbf{A}}_K^+/(\alpha^m,p^n))\otimes_{\ZZ/p^n}-)$ to a finite presentation of $\Lambda_n$, which is exact as $\tilde{\mathbf{A}}_K^+/(\alpha^m,p^n)$ is flat over $\ZZ/p^n$ and by discreteness of all rings, and then conclude with the five lemma.
\end{proof}
\begin{lemma}[{\cite[Lemma 5.18]{marks2023prismatic}}]\label{cover-final-object}
    $(A_\mathrm{inf}(\mathcal{O}_{\hat{K}_\infty}),\ker\theta)$ covers the final object $\ast$ of the topos $\mathrm{Shv}((\mathcal{O}_K)_{\Prism^\circ})$.
    \end{lemma}
\begin{theorem}\label{main-thm2}
    Let $T^\Lambda:=T^\Lambda_{\mathrm{Spf}(\mathcal{O}_K)}$ be the étale realization functor. There are equivalences of categories
    \begin{align*}
        \mathrm{Mod}_{\varphi,\Gamma}^\text{ét}(\fA_{K,\Lambda})\overset{\cong}{\longleftarrow} \mathrm{Vect}((\mathcal{O}_K)_{\Prism^\circ},\hEL)^{\varphi=1}\overset{\cong}{\longrightarrow} \mathrm{Rep}_\Lambda(G_K)
    \end{align*}
    where the first arrow is given by evaluating a $(\Lambda,F)$-crystal at the prism $(\mathbf{A}_K^+, (\varphi^n([p]_q)))$ and the second arrow by étale realization $T^\Lambda$.
\end{theorem}
\begin{proof}
    The second equivalence is Theorem \ref{mainthm-1}. To prove the first equivalence, we argue as in \cite[Theorem 5.2.]{wu2021galois} (see also \cite[Theorem 5.20.]{marks2023prismatic}).
    Since $(A_\mathrm{inf}(\mathcal{O}_{\hat{K}_\infty}),\ker\theta)$ covers the final object of $\mathrm{Shv}((\mathcal{O}_K)_{\Prism^\circ})$, we have  \[
        \mathrm{Vect}((\mathcal{O}_K)_{\Prism^\circ},\hEL)^{\varphi=1}\cong \lim_\bullet \mathrm{Vect}(\hEL(A^\bullet, \ker\theta))^{\varphi=1}
    \]
    where $(A^\bullet, \ker\theta)$ denotes the Čech nerve of $(A_\mathrm{inf}(\mathcal{O}_{\hat{K}_\infty}),\ker\theta)$ in $(\mathcal{O}_K)_\Prism^{\mathrm{perf}}$. Now 
    Lemma \ref{cech-nerve-ainf} shows that the category on the right-hand side is just $\mathrm{Mod}_{\varphi,\Gamma}^\text{ét}(\tilde{\fA}_{K,\Lambda})$.
\end{proof}
\begin{remark}
    One can prove the second equivalence in Theorem \ref{main-thm2} more explicitly and without the use of the second equivalence 
    of Theorem \ref{mainthm-1}; that is, without the use of v-descent.
    Indeed, by the first equivalence of Theorem \ref{mainthm-1}, one may equivalently work over the perfect prismatic site. 
    Define $\tilde{\fA}_\Lambda:=\hEL(A_{\mathrm{inf}}(\mathcal{O}_{\CC_p}), \ker \theta)$ 
    and note that $\tilde{\fA}_\Lambda$ admits a $\Lambda$-linear action of $G_K$, which comes from the usual $G_K$ action on 
    $W(\CC_p^\flat)$. Then, as in \cite[Theorem 5.6.]{wu2021galois}, one checks that $(A_{\mathrm{inf}}(\mathcal{O}_{\CC_p}), \ker \theta)$ 
    covers the final object of the topos $\mathrm{Shv}((\mathcal{O}_K)^{\mathrm{perf}}_\Prism)$ and for the Čech nerve 
    $(A^\bullet, \ker\theta)$ of $(A_\mathrm{inf}(\mathcal{O}_{\CC_p}),\ker \theta)$ in $(\mathcal{O}_K)^{\mathrm{perf}}_\Prism$ it is
    $A^\bullet[1/I]^\wedge_{p}\cong C^\bullet(G_K, W(\CC_p^\flat))$. In particular,  
    $\mathrm{Vect}((\mathcal{O}_K)_{\Prism^\circ},\hEL)^{\varphi=1} \cong {\mathrm{Mod}_{\varphi,G_K}^\text{ét}(\tilde{\fA}_{\Lambda})}$. 
    Since $\tilde{\fA}_\Lambda/\mathfrak{m}$ is a finite product of copies of $\CC_p^\flat$, by \cite[Proposition 3.2.4.]{Schneider_2017} and $\mathfrak{m}$-adic approximation, 
    it follows that taking Frobenius fixed points induces an equivalence ${\mathrm{Mod}_{\varphi,G_K}^\text{ét}(\tilde{\fA}_{\Lambda})}\cong \mathrm{Rep}_\Lambda(G_K)$. Thus, $T^\Lambda$ is given by evaluating a crystal at the prism $(A_{\mathrm{inf}}(\mathcal{O}_{\CC_p}), \ker \theta)$ followed by taking Frobenius fixed points.
\end{remark}
In order to describe the cohomology of $(\Lambda,F)$-crystals on $(\mathcal{O}_K)_{\Prism^\circ}$ in terms of their values on the Čech nerve of 
$(A^\bullet, \ker\theta)$, we make use of the following general fact.
\begin{lemma}\label{cohomology-on-site-cech-complex}
 Let $\mathcal{C}$ be a site admitting finite non-empty products and a weakly final object $X\in \mathcal{C}$ with Čech nerve $X^\bullet$. 
 For any abelian sheaf $\mathcal{F} \in \mathrm{Ab}(C)$, with vanishing higher cohomology on $X^{(n)}$ for all $n\ge 1$, there is a quasi-isomorphism\[
    R\Gamma(\mathcal{C},\mathcal{F})\cong R\lim_\bullet \mathcal{F}(X^\bullet).
 \]
\end{lemma}
\begin{proof}
See \cite[Footnote 10, p. 42]{bhatt2022prisms}.
\end{proof}
\begin{lemma}
    Let $\mathcal{M}\in  \mathrm{Vect}((\mathcal{O}_K)_{\Prism^\circ},\hEL)$ and $(A,I)\in (\mathcal{O}_K)_{\Prism^\circ}$. 
    Then $H^q((A,I),\mathcal{M})=0$ for all $q>0$.
\end{lemma}
\begin{proof}
    By Proposition \ref{vector-bundle-on-hEL} we have $\mathcal{M}_{\mid(A,I)}=\mathcal{M}(A,I)\otimes_{\hEL(A,I)}\hEL$ with 
    $\mathcal{M}(A,I)$ being finite projective over $\hEL(A,I)$. The result then follows from\[
    R\Gamma((A,I),\mathcal{M})\cong R\Gamma((A,I),\hEL)\otimes^{\mathbb{L}}_{\hEL(A,I)}\mathcal{M}(A,I)    
    \]
    and Proposition \ref{hEl-sheaf}.
\end{proof}
\begin{theorem}
Let $\mathcal{M}\in \mathrm{Vect}((\mathcal{O}_K)_{\Prism^\circ},\hEL)^{\varphi=1}$, $V=T^\Lambda(\mathcal{M})$ and 
$M:=\mathcal{M}(\mathbf{A}_K^+, (\varphi^n([p]_q)))$ the corresponding $(\varphi,\Gamma)$-module over $\mathbf{A}_{K,\Lambda}$.
There are quasi-isomorphisms
\[
   C(\Gamma^\bullet,M)^{\varphi=1} \cong R\Gamma((\mathcal{O}_K)_{\Prism^\circ},\mathcal{M})^{\varphi=1} \cong  C(G_K^{\bullet},V).
\]
\end{theorem}
\begin{proof}
The second quasi-isomorphism is Theorem \ref{prismatic-proetale-comparison}. The first quasi-isomorphism follows from Lemma \ref{cech-nerve-ainf} and Lemma \ref{cohomology-on-site-cech-complex} and the fact that 
$C(\Gamma^\bullet, M)^{\varphi=1}\cong C(\Gamma^\bullet,M\otimes_{\mathbf{A}_{K,\Lambda}}\tilde{\mathbf{A}}_{K,\Lambda})^{\varphi=1}$, which follows from the full faithfulness in Lemma \ref{perfection-same-phi-modules}.
\end{proof}
\cleardoublepage
\phantomsection
\addcontentsline{toc}{section}{References}
\printbibliography

\textsc{Ruprecht-Karls-Universität Heidelberg, Mathematisches Institut, Im Neuenheimer Feld 205, D-69120 Heidelberg}

\textit{E-mail address}: \texttt{mschneider@mathi.uni-heidelberg.de}
\end{document}